\documentclass{article}

\usepackage{amsthm}
\usepackage{amssymb}
\usepackage{fullpage}
\usepackage{amsmath}
\usepackage{graphicx}
\usepackage{tikz}
\usepackage{mathrsfs}
\usepackage{float}
\usepackage{makecell}
\usepackage{amscd}
\usepackage{multirow}
\usepackage{hhline}
\usepackage{verbatim}
\usepackage{mathtools}
\usepackage{xspace}
\usepackage{tikz-cd}

\newcommand{\Z}{\mathbb{Z}}
\newcommand{\E}{\mathbb{E}}
\newcommand{\N}{\mathbb{N}}

\newcommand{\Q}{\mathbb{Q}}

\renewcommand{\Im}{\text{Im}}
\renewcommand{\epsilon}{\varepsilon}

\newcommand{\biglink}{\bigcap_{i = 0}^{d - 1} \link([x_0, ..., \hat{x_i}, ..., x_{d - 1}], Y)}

\DeclareMathOperator{\link}{link}
\DeclareMathOperator{\del}{del}

\newcommand\polymake{\texttt{polymake}\xspace}

\usepackage{relsize}

\usepackage{caption}
\newcounter{Results}[section] 

\usepackage[colorinlistoftodos,bordercolor=orange,backgroundcolor=orange!20,linecolor=orange,textsize=scriptsize]{todonotes}

\newtheorem{theorem}[Results]{Theorem}
\newtheorem{lemma}[Results]{Lemma}
\newtheorem{proposition}[Results]{Proposition}
\newtheorem{corollary}[Results]{Corollary}

\newtheorem{claim}[Results]{Claim}

\newtheorem*{exercise*}{Exercise}

\theoremstyle{remark}
\newtheorem{remark}[Results]{Remark}
\newtheorem*{remark*}{Remark}

\theoremstyle{definition}
\newtheorem{definition}[Results]{Definition}

\newtheorem*{definition*}{Definition}
\newtheorem*{example*}{Example}
\newtheorem*{notation}{Notation}

\numberwithin{equation}{section}

\title{The worst way to collapse a simplex}

\author{Davide Lofano \thanks{Technische Universit\"{a}t Berlin,  Chair of Discrete Mathematics/Geometry,  Strasse des 17.~Juni 136, 10623 Berlin, Germany. Supported by the Deutsche Forschungsgemeinschaft (DFG, German Research Foundation) Graduiertenkolleg ``Facets of Complexity" (GRK 2434).} \and Andrew Newman \footnotemark[1]}

\begin{document}
\maketitle
\abstract
 In general a contractible simplicial complex need not be collapsible. Moreover, there exist complexes which are collapsible but even so admit a collapsing sequence where one ``gets stuck", that is one can choose the collapses in such a way that one arrives at a nontrivial complex which admits no collapsing moves. Here we examine this phenomenon in the case of a simplex. In particular we characterize all values of $n$ and $d$ so that the $n$-simplex may collapse to a $d$-complex from which no further collapses are possible. Equivalently, and in the language of high-dimensional generalizations of trees, we construct hypertrees that are anticollapsible, but not collapsible.

\section{Introduction}
A standard notion in computational topology is that of collapsibility, first introduced by Whitehead \cite{whitehead1939simplicial}. In the setting of simplicial complexes, collapsibility is a completely combinatorial notion. Recall that an abstract simplicial complex on a ground set $V$ is a collection of sets, called faces, in $2^V$ that is closed under taking subsets. For a simplicial complex $X$, a nonempty face of $X$ is said to be \emph{free} provided that is a proper subset of exactly one other face in $X$. An $\emph{elementary collapse}$ of $X$ is the process of removing a free face and the unique face properly containing it. 

We overview the relevant topological notions in Section \ref{Sec:Preliminaries}. Generally speaking though, elementary collapses preserve many topological invariants that one might want to compute for some simplicial complex. The strategy from a computational topology perspective, therefore, becomes to use elementary collapses to reduce some initial complex $X$ to a smaller complex $Y$ where the relevant topological invariants may be more easily computed. The best-case scenario for such an approach is when $X$ is \emph{collapsible}.  A simplicial complex is said to be collapsible if there exists a sequence of elementary collapses that reduce the complex to a single vertex.

The notion of collapsibility is contrasted with the purely topological notion of contractibility. It is easy to check that contractibility implies collapsibility, but the converse is not true in general. The smallest such example is the dunce hat \cite{zeeman1963dunce}, which we will discuss in more detail below, a 2-dimensional contractible complex with triangulations on 8 vertices which are not collapsible. For smaller simplicial complexes a result of \cite{bagchi2005combinatorial} shows that for simplicial complexes on 7 or fewer vertices contractibility and collapsibility are equivalent. 

While collapsibility and contractibility are not equivalent in general, more subtlety there even exists simplicial complexes which are collapsible, but for which it is possible to choose a sequence of elementary collapses after which one gets stuck at a nontrivial complex that is not collapsible. Formally, we say that a collapsing sequence on a complex $X$ gets stuck at a complex $Y$, if the collapsing sequence reduces $X$ to $Y$ and $Y$ has no free faces. More broadly, we say that a collapsing sequence on $X$ gets stuck at dimension $d$ if the collapsing sequence gets stuck at some $d$-dimensional complex.

There are examples of this even in the case of a simplex. Recall that the simplex on $n$ vertices, denoted $\Delta_{n - 1}$ is the simplicial complex on ground set $\{1, .., n\}$ where every subset of $\{1, ..., n\}$ is a face. It is easy to see that $\Delta_{n - 1}$ is collapsible. Indeed one may choose a vertex and proceed dimension-wise, collapsing at each step the free maximal faces that do not contain the chosen vertex. However, \cite{benedetti2009dunce} shows that there exists a sequence of elementary collapses from the 7-simplex to a triangulation of the dunce hat, so from there no further collapses are possible even though the 7-simplex is collapsible. This notion is central in the implicit question in the title of our paper, what is the worst way to collapse a simplex? That is starting from the simplex $\Delta_{n - 1}$ what is the maximal dimension $d$ in which a collapsing sequence may get stuck? We answer this question with the following theorem:

\begin{theorem}\label{theorem2}
For $n \geq 8$ and $d \notin \{1, n - 3, n - 2, n - 1\}$, there exists a collapsing sequence of the simplex on $n$ vertices which gets stuck at a $d$-dimensional complex on $n$ vertices. Moreover for $n \leq 7$ and $d$ arbitrary or $n$ arbitrary and $d \in \{1, n - 3, n - 2, n - 1\}$ it is not possible to find a collapsing sequence of the simplex on $n$ vertices which gets stuck at dimension $d$. 
\end{theorem}

Theorem \ref{theorem2} implies the following corollary. 

\begin{corollary}\label{theorem1}
For every $n \geq 8$ and $d$ with $d \notin \{1, n - 3, n - 2, n - 1\}$ there exists a contractible $d$-complex on $n$ vertices with no free faces. Moreover this is best possible, for $n \leq 7$ or $d \in \{1, n - 3, n - 2, n - 1\}$ every contractible $d$-complex on $n$ vertices has a free face.
\end{corollary}

Strictly speaking only the first part of Corollary \ref{theorem1} is implied by Theorem \ref{theorem2} and not the second part. However our proof of the second part of Theorem \ref{theorem2} proceeds by showing that for $n$ arbitrary and $d \in \{1, n - 3, n - 2, n - 1\}$ every contractible complex on $n$ vertices of dimension $d$ has a free face and for $n \leq 7$ we cite the previously mentioned result of  \cite{bagchi2005combinatorial}.


Another way to state Theorem \ref{theorem2} is in terms of \emph{anticollapsibility}. An elementary anticollapse is the reverse of an elementary collapse, we define this formally in Section \ref{section:DMT}, and we say that a complex $X$ on $n$ vertices is anticollapsible provided that there exists a sequence of anticollapsing moves from $X$ to $\Delta_{n - 1}$. Equivalently, there is a standard notion of Alexander duality for simplicial complexes, which we also discuss below, and $X$ is anticollapsible if and only if its Alexander dual is collapsible. In terms of anticollapsibility, Theorem \ref{theorem2} tells us that for $n \geq 8$ and $d \notin \{1, n - 3, n - 2, n - 1\}$, there exists a $d$-dimensional simplicial complex on $n$ vertices which is anticollapsible, but has no free faces. 

Within this framework of constructing complexes which satisfy some specified nonempty subset of the conditions of contractibility, collapsibility, and anticollapsibility, we are really considering properties of higher-dimensional generalizations of trees. For 1-dimensional complexes contractibility, collapsiblity, and anticollapsiblity are all equivalent and hold exactly for trees. Each of these properties may be generalized to higher dimensions to describe high-dimensional trees, but the properties are no longer equivalent. In general either collapsibility or anticollapsiblity will imply contractibility, but no other implications hold. High-dimensional trees have been studied from various viewpoints by, for example \cite{LinialPeled2019, Kalai1983, DuvalKlivansMartin2009, KahnSaksSturevant1984, adiprasito2017extremal, BjornerLutz2000, KLNP}. In Section \ref{section:hypertrees} we overview some of the literature on this topic, but for now we point out that our results in particular characterize values of $n$ and $d$ so that there are $d$-dimensional trees on $n$ vertices which are anticollapsible but not collapsible.

\section{Preliminaries}\label{Sec:Preliminaries}
We start with some standard notions in combinatorial topology, in particular notions from simplicial homology. Much more on these notions can be found in a standard reference such as \cite{Hatcher}. If $X$ is a finite simplicial complex and $\tau$ is a face of $X$, the \emph{dimension of $\tau$}, denoted $\dim \tau$, is $|\tau| - 1$, and the \emph{dimension of $X$}, denoted $\dim X$, is the maximum dimension of any face of $X$. A maximal face is called a \emph{facet}. A complex $X$ is said to be \emph{pure dimensional} if all of its facets are of the same dimension. The \emph{$j$-skeleton} of $X$, denoted $X^{(j)}$ is the subcomplex of $X$ given by every face of $X$ of dimension at most $j$; we say that the $j$-skeleton of $X$ is complete provided $X^{(j)}$ has all possible $j$-dimensional faces on the vertices of $X$. 

We are now ready to define simplicial homology, though as we define it here we will actually be describing \emph{reduced} simplicial homology, but this makes things easier to describe and works with the present setting. We begin with a simplicial complex $X$ and an abelian group of coefficients $R$; most commonly $R$ will be taken to be $\Z$, $\Q$ or a prime-order finite field. For each integer $i $ the set of $i$-chains denoted $C_i$ is the family of formal linear combinations of faces of $X$ of dimension $i$ with coefficients in $R$ (and $C_{-1}=R$). For each $i \geq 0$, one defines the $i$th boundary map to be the linear map $\partial_i : C_i \rightarrow C_{i - 1}$ given by sending the generator $[x_0, ..., x_i]$ to $\sum_{j = 0}^{i} (-1)^i [x_0, ..., \hat{x_j}, .., x_i]$ where $[x_0, ..., \hat{x_j}, .., x_i]$ denotes the face given by deleting $x_j$ from $[x_0, .., x_i]$. The $i$th homology group with coefficients in $R$ denoted $H_i(X; R)$ is defined to be the abelian group given by the quotient $\ker(\partial_i)/\Im(\partial_{i + 1})$ (it is routine to check that this quotient is well-defined). Without specifying the ring, the $i$th homology group of $X$, $H_i(X)$ is taken to be $H_i(X; \Z)$. A standard fact about homology is that $H_0(X) = 0$ if and only if $X$ is path-connected. In fact $H_0(X)$ is the free abelian group with rank equal to the number of connected components of $X$ minus one. It is also clear from the definition that homology groups vanish in dimensions larger than the dimension of $X$.

A complex is said to be $R$-acyclic if $H_i(X; R) = 0$ for all $i \geq 0$. Regarding the cases $R = \Q$ and $R = \Z$, we have that if $X$ is $\Z$-acyclic then it is necessarily $\Q$-acyclic, however the converse need not hold. A $\Q$-acyclic complex will necessarily have all homology groups finite, but they will not necessarily be trivial; this follows by the standard universal coefficient theorem for homology. For 1-dimensional complexes (i.e. for graphs), however $\Z$-acyclic and $\Q$-acyclic are equivalent and hold precisely when the graph is a tree.

Another notion we have not yet defined is that of contractibility. To be properly defined for a topological space, one has to begin with defining homotopy equivalence. For our purposes here we omit some of the technicalities and only say that it is standard that if two complexes are homotopy equivalent they have the same homology groups. For the most part all homotopy equivalences here will be \emph{simple-homotopy equivalences}. Two simplicial complexes $X$ and $Y$ are said to be simple-homotopy equivalent provided there is a sequence of collapses and anticollapses from $X$ to $Y$. A space is said to be contractible if it is homotopy equivalent to a point. For our purposes here we will often verify contractibility of complexes by checking anticollapsibility. An anticollapsible complex is simple-homotopy equivalent to a simplex, and a simplex is collapsible, so it is simple-homotopy equivalent to a point. 

\subsection{Discrete Morse Theory}\label{section:DMT}

We recall here the main concepts of Forman's discrete Morse theory \cite{forman1998morse,forman2002user}.
We follow the point of view of Chari \cite{chari2000discrete} and the book by Kozlov \cite{kozlov2007combinatorial}, using acyclic matchings instead of discrete Morse functions.

Let $X$ be a simplicial complex, $P$ be the poset of faces of $X$, and $G$ the Hasse diagram of $P$, i.e.\ the graph with vertex set the simplices of $X$ and having an edge $(\tau, \sigma)$ whenever $\tau \subset \sigma$ and $\dim \sigma = \dim \tau +1$. Moreover let us denote by $E$ the set of edges of $G$.

Given a subset $M$ of $E$, we can orient all edges of $G$ in the following way: an edge $(\tau,\sigma) \in E$ is oriented from $\tau$ to $\sigma$ if the pair does not belong to $M$, otherwise in the opposite direction.
Denote this oriented graph by $G_{M}$.

\begin{definition}[Acyclic matching \cite{chari2000discrete}]
  A \emph{matching} on $G$ is a subset $M \subseteq E$ such that every face of $X$ appears in at most one edge of $M$.
  A matching $M$ is \emph{acyclic} if the graph $G_{M}$ has no directed cycle.
\end{definition}

Given a matching $M$ on $G$, an \emph{alternating path} is a directed path in $G_{M}$ such that two consecutive edges of the path do not belong both to $M$ or both to $E \setminus M$.
The faces of $X$ that do not appear in any edge of $M$ are called \emph{critical} with respect to the matching $M$.

We are now ready to state the main theorem of discrete Morse theory. 

\begin{theorem}[\cite{forman1998morse,chari2000discrete}]
  Let $X$ be a simplicial complex, and let $P$ be its poset of faces, and $G$ the Hasse diagram of $P$.
  If $M$ is an acyclic matching on $G$, then $X$ is homotopy equivalent to a CW complex $X_M$, called the \emph{Morse complex} of $M$, with cells in dimension-preserving bijection with the critical cells of~$X$. Furthermore if the critical cells forms a subcomplex $X_c$ of $X$, then there exists a sequence of elementary collapses collapsing $X$ to $X_c$.
\end{theorem}

One key application of this theorem is in computational topology. Indeed it is often the case that discrete Morse theory gives us a way to find a homotopy equivalence between a given simplicial complex and a much smaller CW-complex on which the homology groups are easier to compute. 

%

Finally, we recall the following standard tool to construct acyclic matchings.

\begin{theorem}[Patchwork theorem {\cite[Theorem 11.10]{kozlov2007combinatorial}}]\label{teo:patchwork}
  Let $P$ be the poset of faces of a simplicial complex $X$, and let $\varphi \colon P \to Q$ be a poset map.
  For all $q \in Q$, assume to have an acyclic matching $M_q \subseteq E$ that involves only elements of the subposet $\varphi^{-1}(q) \subseteq P$.
  Then the union of these matchings is itself an acyclic matching on $P$.
\end{theorem}

Another way to look at acyclic matchings is through the notion of collapsibility. As already said, an \emph{elementary collapse} is simply the removal of two simplices $\sigma$ and $\tau$ such that
\begin{itemize}
 \item $\dim \sigma = \dim \tau +1$,
 
 \item the only simplex containing $\sigma$ is $\sigma$ itself,
 
 \item the only simplices containing $\tau$ are $\sigma$ and $\tau$, in which case $\tau$ is called a \emph{free face}.
\end{itemize}

We say that a complex is \emph{collapsible} if, through a series of elementary collapses, it can be reduced to a single vertex. As a weaker notion we say that a complex is $d$-collapsible provided it can be reduced to a complex of dimension less than $d$ by a sequence of elementary collapses.

An \emph{elementary anticollapse} \cite{cohen2012course}, sometimes also called expansion, is the dual operation, i.e. the gluing of two simplices $\sigma'$ and $\tau'$ such that
\begin{itemize}

 \item $\dim \sigma' = \dim \tau' +1$,
 
 \item $\tau'$ is not in $X$,
 
 \item the only facet of $\sigma'$ not contained in $X$ is $\tau'$.
\end{itemize}

 If $X$ is on $n$ vertices we say that it is anticollapsible if, through a series of elementary anticollapses, it can be expanded to the simplex on $n$ vertices. As a weaker notion a complex is $d$-anticollapsible if there is a sequence of anticollapses so that the resulting complex has complete $d$-skeleton. It should be noted that, while it is always possible to perform an elementary anticollapse that adds a new vertex, we are prohibiting these moves while talking about anticollapsibility. 

The combinatorial encoding of a set of collapses is best provided by a matching consisting of a collection of pairs of cells $(\tau,\sigma)$, which it is easy to prove to be an acyclic matching. In fact, a standard algorithm to find an acyclic matching is the Random Discrete Morse algorithm of Benedetti and Lutz \cite{benedetti2014random} which uses elementary collapses. This algorithm is 	implemented in \polymake \cite{polymake:2000}. The Random Discrete Morse algorithm begins with a simplicial complex $X$ and performs elementary collapses at random dimension-by-dimension. If the collapsing sequence reaches a point where there are no free faces, a top-dimensional face is selected uniformly at random, marked as a critical cell and deleted. Then this same procedure is performed until only a single vertex remains. The algorithm returns the number of critical cells in each dimension; the list of critical cells in each dimension of a complex with an acyclic matching is called a \emph{Morse vector}. A complex will have a Morse vector $(1, 0, \ldots , 0)$ with respect to some acyclic matching if and only if it is collapsible, so the Random Discrete Morse algorithm may be used to verify collapsibility. 

In at least one particular case the Random Discrete Morse algorithm may also be used to verify that a complex is not collapsible. If a complex is $d$-dimensional and the Random Discrete Morse algorithm returns a Morse vector with at least one critical cell in dimension $d$ then the complex is not collapsible, in fact in such a case the complex is not even $d$-collapsible. This is implied by the following definition and lemma. 

\begin{definition}
A $d$-dimensional simplicial complex $L$ is called a \emph{core} provided that every $(d - 1)$-dimensional face of $L$ is contained in at least two $d$-dimensional faces of $L$.
\end{definition}

\begin{lemma}\label{Lemma:core}
Let $X$ be a $d$-dimensional complex. $X$ is $d$-collapsible if and only if $X$ does not contain a $d$-dimensional core.
\end{lemma}
\begin{proof}
Suppose that $X$ is not $d$-collapsible, then there exists some sequence of collapses that results in a complex $X'$ that is still $d$-dimensional, but has no collapsing moves possible. Now in this complex every $(d - 1)$-dimensional face is contained in either zero $d$-dimensional faces, in which case it is said to be isolated, or at least two $2$-dimensional faces. Thus the pure part of this complex, that is the subcomplex whose facets are exactly the $d$-dimensional faces, is a core. 

Conversely, suppose that $X$ is $d$-collapsible and contains a core $Y$. Then there exists a sequence of collapses of removing pairs $(\tau, \sigma)$ with $\tau$ of dimension $d - 1$ and $\sigma$ of dimension $d$ that reduce $X$ to a $(d - 1)$-dimensional complex. Let $(\tau, \sigma)$ be the first such pair in the sequence with $\sigma \in Y$. Since $Y$ is a subcomplex, $\tau$ also belongs to $Y$. At the moment we collapse $(\tau, \sigma)$, $\tau$ has degree 1 in $X$, and hence it has degree 1 in $Y$. But then by definition of a core $\tau$ is contained in $\sigma' \neq \sigma$ so that $\sigma' \in Y$. It follows, however, that $\sigma'$ must have been removed from $X$ before $\sigma$, but this contradicts the choice of $\sigma$. 
\end{proof}

If a complex is $d$-dimensional and the Random Discrete Morse algorithm returns a Morse vector with at least one critical cell in dimension $d$, then the complex contains a $d$-dimensional core. Indeed this means that the random collapsing sequence reached a point where the complex was $d$-dimensional but had no free faces. By Lemma \ref{Lemma:core} this means the complex is not $d$-collapsible.

\subsection{Alexander duality and the top dimensions}
Given a simplicial complex $X$, there is a natural way to define an Alexander dual $X^*$. Here we give the definition and main theorem for this duality as described in \cite{BjornerTancer}. 

Let $X$ be a simplicial complex having vertex set $V$. Given a subset $\sigma \subseteq V$ let $\sigma^c=V \setminus \sigma$ denote the complementary vertex set. 

\begin{definition}
 The Alexander dual of $X$ on $V$ is the simplicial complex defined by
 \begin{equation*}
  X^*:=\{ \sigma \subseteq V \mid \sigma^c \notin X \}.
 \end{equation*}

\end{definition}

 It is easy to see that $X^{**}=X$. Furthermore, for simplicial complexes we have the following notion of combinatorial Alexander duality similar to classic Alexander duality for more general topological spaces. 
 
 \begin{theorem}[Combinatorial Alexander duality \cite{Kalai1983}]\label{thm:AlDual}
  Let $X$ be a simplicial complex on $n$ vertices. Then
  \begin{equation*}
   {H}_i(X) \cong {H}^{n-i-3}(X^*).
  \end{equation*}
  where $H^{n - i - 3}(X^*)$ is the $(n - i - 3)$rd cohomology group of $X$.
 \end{theorem}
 We don't formally define (reduced) cohomology here, but it suffices for our purposes to just mention that if $X$ is acyclic then all its cohomology groups vanish and that $H^0(X) = 0$ if and only if $X$ is connected.
 
The Alexander dual behaves exceptionally well with respect to collapsibility, indeed the dual of an elementary collapse in $X$ is an elementary anticollapse in $X^*$. This is standard to check, but we prove it in Proposition \ref{Prop:Duality}.

\begin{notation}
Before proving this and later results, we introduce some notations that we use throughout the paper:
 \begin{itemize}
  \item We use parentheses when talking about a set of vertices, e.g. $V=(x_0,x_1,\ldots,x_n)$.
  
  \item We use square brackets when talking about a face of a simplicial complex, e.g. $\sigma=[x_0,x_2]$. With a little abuse of notation if $x_k \notin \sigma$ we will denote by $[x_k,\sigma]$ the face with vertex set $x_k$ and the vertices of $\sigma$, and if $\tau,\sigma$ are both faces $[\tau,\sigma]$ the face with vertex set the union of the vertices of $\sigma$ and $\tau$.
  
  \item We use curly brackets to denote a simplicial complex; we will use the same notation both to list all the faces or only the facets, which one we are using will be clear from the context; e.g. $X=\{\emptyset, [x_0],[x_1],[x_0,x_1]\}$ or $X=\{[x_0,x_1]\}$.
 \end{itemize}

\end{notation}

\begin{proposition} \label{Prop:Duality}
 If $X$ collapses to $Y$ then $X^*$ anticollapses to $Y^*$. In particular if $X$ is collapsible then $X^*$ is anticollapsible.
\end{proposition}
\begin{proof}
Let $X$ be a simplicial complex on the ground set $V$ with $|V| = n$. We show that an elementary collapse on $X$ corresponds to an elementary anticollapse on $X^*$. Suppose that $\tau = [x_0, \dots, x_k]$ is free in $X$ with unique coface $\tau' = [x_0, \dots, x_k, x_{k + 1}]$ and we perform the elementary collapse removing $\tau$ and $\tau'$ to arrive at $X'$. We show that $(X')^*$ is obtained from $X^*$ by an elementary anticollapse. The claim will then follow by induction.

Since $X' = X \setminus \{\tau, \tau'\}$, we have that $(X')^* = \{\sigma \subseteq V \mid \sigma^c \notin X \} \cup \{\tau^c, (\tau')^c\} = X^* \cup \{\tau^c, (\tau')^c\}$. Thus $(X')^*$ is obtained from $X^*$ by adding $\tau^c = V \setminus [x_0, \dots, x_k]$ and $(\tau')^c = V \setminus [x_0, \dots, x_{k + 1}]$. Therefore $(\tau')^c$ is an $(n - k - 3)$-simplex and $\tau^c$ is an $(n - k - 2)$-simplex with $(\tau')^c \subseteq \tau^c$. Moreover since $\tau, \tau' \in X$, we have that $(\tau')^c, \tau^c$ do not belong to $X^*$. Thus we only have to check that all of the facets of $\tau^c$ different from $(\tau')^c$ are contained in $X^*$. Let $\sigma \subseteq \tau^c$ and suppose that $\sigma \notin X^*$, then $\sigma^c \in X$ and $\tau \subseteq \sigma^c$, but since $\tau$ is free we have that $\sigma^c = \tau'$. Thus $(X')^*$ is obtained from $X^*$ by an elementary collapse. 
\end{proof}

\begin{remark}
To be completely precise when discussing duality and collapsibility we have to allow for the trivial collapse of the empty set as a free face of a simplicial complex with only one vertex. Typically, this case is not considered when discussing collapsible complexes, but observe that the dual of the simplex on $n$ vertices is the empty simplicial complex on ground set $[n]=\{1,\dots,n\}$. Nonetheless, even if we do not allow the trivial collapse the second part of Proposition \ref{Prop:Duality} remains true as the dual of a complex on the ground set $[n]$ with only one vertex is the boundary of the simplex on $n$ vertices with a single $(n - 2)$-dimensional face removed, and this anticollapses in one step to the simplex, with this step dual to the trivial collapse.
\end{remark}

As an application of Alexander duality we see the following proof which is a simple generalization of what was done in 
\cite{bagchi2005combinatorial} in the case of $7$ vertices.

\begin{proposition}\label{Prop:NonExistence}
 Any contractible simplicial complex on $n$ vertices of dimension larger or equal to 
$n-3$ must have at least one free face. 
\end{proposition}

\begin{proof}
 
The only simplicial complex of dimension $n-1$ on $n$ vertices is the $(n-1)$-simplex, which clearly has every one of its $(n-2)$-dimensional faces free. And of 
course there are no complexes of dimension bigger than $n-1$ on $n$ vertices.

Let us assume that $X$ is a contractible simplicial complex on $n$ vertices and of 
dimension $n-2$. Then any $(n-3)$-dimensional face  has $n-2$ vertices, so can be contained only in 
$0, 1$ or $2$ $(n-2)$-dimensional faces of $X$. If all the $(n-3)$-dimensional faces are contained in $0$ or $2$ $(n-2)$-dimensional faces then the 
union of all the $(n - 2)$-dimensional faces yields a cycle in the degree $(n-2)$ homology group with $\Z/2\Z$-coefficients which is 
impossible since the complex is contractible. Then we have at least one free $(n - 3)$-dimensional face.

The remaining case is when the complex $X$ is $(n-3)$-dimensional. In this case we can 
look at the Alexander dual $X^*$ of $X$. By Combinatorial Alexander duality \ref{thm:AlDual}, $X^*$ will be a connected complex on $n$ vertices since $X$ is contractible. Moreover, since $X$ is $(n - 3)$-dimensional, some edges of $X^*$ is missing. Thus there exists vertices $x$, $y$, and $z$ in $X^*$ so that $[x, z], [y, z] \in X^*$, but $[x, y] \notin X^*$. But this implies that the simplex $[x,y,z]^c$ is a free $(n - 4)$-dimensional face of $X$. 
\end{proof}

\begin{corollary}\label{Cor:TopDim}
 Given a simplicial complex $X$ on $n$ vertices there does not exist a collapsing sequence which gets stuck at dimension at least $n - 3$. 
\end{corollary}

%
%

\subsection{Hypertrees}\label{section:hypertrees}
In the 1-dimensional case, a tree is characterized as a graph which is connected and has no cycles. Thus in the language of homology, a graph $G$ is a tree if and only if $H_1(G) = H_0(G) = 0$. This definition was extended by Kalai in \cite{Kalai1983} to the higher-dimensional notion of a $\Q$-acyclic complex; a $\Q$-acyclic complex $X$, also called a hypertree, is a simplicial complex so that $H_i(X; \Q) = 0$ for all $i \geq 0$. 

We point out that in Kalai's original formulation a $d$-dimensional $\Q$-acyclic complex was defined to have complete $(d - 1)$-skeleton. This will not be a requirement for our complexes, though in many of our constructions it will hold and gives an easy way to check that homology vanishes in degrees below $d - 1$.

One way in which $d$-dimensional hypertrees are more interesting that 1-dimensional trees is in the ways that certain equivalent properties for trees generalize to nonequivalent properties for hypertrees. In particular there is the following chain of implications for properties of a hypertree $X$.

\begin{center}
\begin{tikzcd}[column sep=tiny, row sep = large]
& \text{$X$ is non-evasive} \arrow[d] \\
& \text{$X$ is collapsible and anticollapsible} \arrow[dl] \arrow[dr] & \\
\text{$X$ is collapsible} \arrow[dr] & & \text{$X$ is anticollapsible} \arrow[dl] \\
& \text{$X$ is contractible}  \arrow[d]\\
& \text{$X$ is $\Z$-acyclic} \arrow[d] \\
& \text{$X$ is $\Q$-acyclic}
\end{tikzcd}
\end{center}
Non-evasiveness, which we have not yet defined, was first described in \cite{KahnSaksSturevant1984}, and, among other definitions, has the following nice inductive one: 
\begin{itemize}
\item A single vertex is non-evasive.
\item $X$ is non-evasive if and only if there exists a vertex $v$ of $X$ so that both $\link(v, X)$ and $\del(v, X)$ are non-evasive, where
\begin{align*}
 \link(v,X) &=  \{ \sigma \in X \mid v \notin \sigma, \ [v, \sigma] \in X \},\\
\del(v,X) &=  \{ \sigma \in X \mid v \notin \sigma \}.
\end{align*}
\end{itemize}
The inductive definition of non-evasiveness may be used to show the first implication and the others are obvious. Moreover, any tree is non-evasive as any one of its leaves has single vertex link and in the $d = 1$ case a $\Q$-acyclic complex is a tree. So it is clear that all the above properties are equivalent for $d = 1$ and hold exactly for trees.

On the other hand, for $d \geq 2$, none of the implications are reversible in general. For the bottom implication in the chain, recall by the universal coefficient theorem that a complex is $\Z$-acyclic if and only if it is $\Q$-acyclic and $\Z/q\Z$ acyclic simultaneously for every prime $q$. Thus, for example, any triangulation of the projective plane is $\Q$-acyclic but not $\Z/2\Z$-acyclic, so hence not $\Z$-acyclic. The standard such triangulation is given by identifying antipodal faces of the icosahedron to produce a triangulation of the projective plane with 6 vertices, 15 edges, and 10 triangles. 

Continuing from bottom to top,  \cite{BjornerLutz2000} gives an example of a $\Z$-acyclic complex which is not contractible. Any triangulation of the dunce hat gives an example of a contractible, but not collapsible 2-complex. Such triangulations are given by \cite{benedetti2009dunce,zeeman1963dunce}. The dual of a triangulated dunce hat gives an example of a contractible, but not anticollapsible complex. The example of \cite{benedetti2009dunce}, that one can get stuck in dimension 2 when collapsing the $7$-simplex, gives an example of an anticollapsible complex which is not collapsible. The dual of such a complex shows also that there are collapsible complexes that are not anticollapsible. Finally, \cite{adiprasito2017extremal} gives an example of a complex which is evasive, but is nonetheless anticollapsible and collapsible.

\section{Constructions}

\begin{figure}[h]
\centering
 \includegraphics[width=8cm]{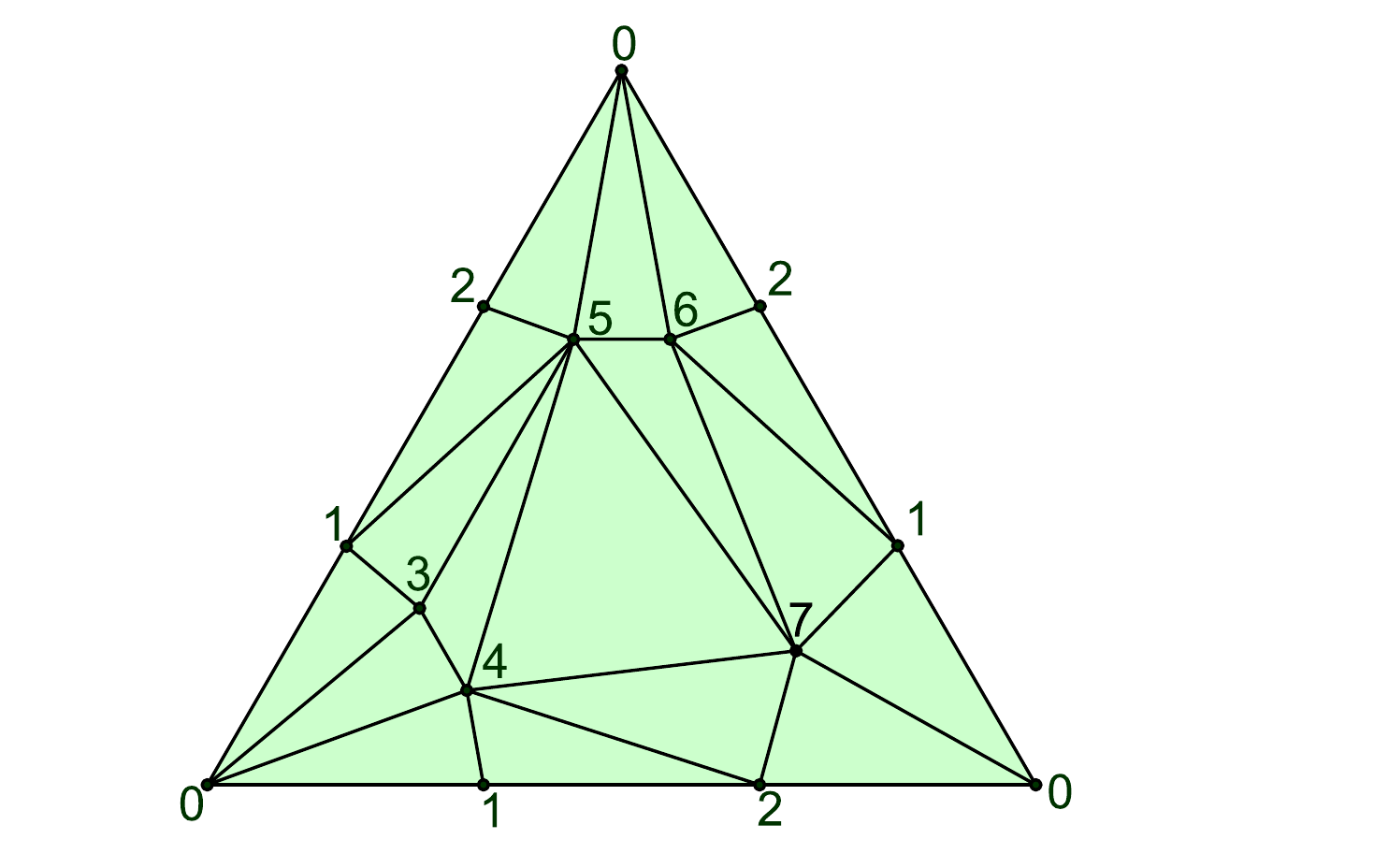}
 \hspace{-2.5em} \caption{ A dunce hat triangulation.}
 \label{fig:1}
\end{figure}

\subsection{Explicit constructions for $n = 8$}\label{sec:Dim8}
Probably the best known example of a contractible but non-collapsible complex is the dunce hat \cite{zeeman1963dunce}, which is known to have triangulations with 8 
vertices. 
Benedetti and Lutz \cite{benedetti2009dunce} presented an $8$-vertex triangulation  of the 
dunce hat, see Figure \ref{fig:1}, that can be found as a subcomplex of, and anticollapses to, a non-evasive ball with $8$ vertices. This in particular implies that this triangulation is anticollapsible.

By Proposition \ref{Prop:NonExistence} we know that any contractible simplicial complex on $8$ vertices in dimension bigger than four has at least one free face. For $d = 3$ and $d = 4$ we considered the dual problem. We looked for 3-dimensional hypertrees and 2-dimensional hypertrees which are collapsible but have no anticollapsing moves possible and found the  following examples given as a list of facets.
\smallskip

$Y_8^2$ := \{ [ 1, 2, 3 ], [ 1, 3, 4 ], [ 1, 4, 5 ], [ 1, 5, 6 ], [ 1, 3, 8 ], [ 1, 6, 8 ], [ 1, 7, 8 ], [ 2, 3, 7 ], [ 3, 4, 6 ], [ 2, 4, 6 ], [ 2, 5, 8 ], [ 2, 6, 7 ], [ 2, 7, 8 ], [ 3, 4, 7 ], [ 3, 5, 7 ], [ 3, 5, 8 ], [ 4, 5, 8 ], [ 4, 6, 8 ], [ 4, 7, 8 ], [ 5, 6, 7 ], [ 1, 2, 6 ] \};

\medskip

$Y_8^3$ := \{ [ 4, 6, 7, 8 ], [ 2, 5, 7, 8 ], [ 1, 5, 7, 8 ], [ 3, 4, 7, 8 ], [ 2, 4, 7, 8 ], [ 2, 3, 7, 8 ],
  [ 1, 3, 7, 8 ], [ 2, 5, 6, 8 ], [ 3, 4, 6, 8 ], [ 1, 4, 6, 8 ], [ 2, 3, 6, 8 ], [ 1, 3, 6, 8 ],
  [ 3, 4, 5, 8 ], [ 2, 4, 5, 8 ], [ 1, 3, 5, 8 ], [ 1, 2, 5, 8 ], [ 2, 3, 4, 8 ], [ 1, 2, 4, 8 ],
  [ 4, 5, 6, 7 ], [ 3, 5, 6, 7 ], [ 2, 5, 6, 7 ], [ 1, 4, 6, 7 ], [ 1, 3, 6, 7 ], [ 1, 2, 6, 7 ],
  [ 1, 4, 5, 7 ], [ 1, 3, 4, 7 ], [ 1, 2, 4, 7 ], [ 3, 4, 5, 6 ], [ 1, 4, 5, 6 ], [ 2, 3, 5, 6 ],
  [ 2, 3, 4, 6 ], [ 1, 2, 4, 6 ], [ 1, 3, 4, 5 ], [ 1, 2, 3, 5 ], [ 1, 2, 3, 4 ] \};

\smallskip

\begin{figure}[h]
\centering
 \includegraphics[width=6cm]{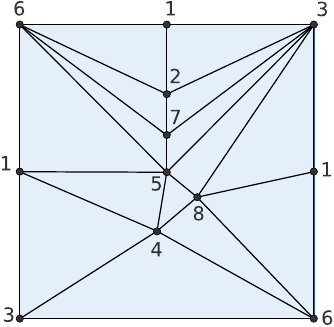}
 \hspace{-2.5em} \caption{ An intermediate step in the construction of $Y_2^8$.}
 \label{fig:3}
\end{figure}

The example $Y_8^2$ was constructed by hand. For $Y_8^3$ we implemented the following generalization of Kruskal's algorithm to generate $d$-dimensional hypertrees on $n$ vertices and checked collapsiblitiy and $d$-anticollapsiblity: 

\begin{enumerate}
	\item Begin with the complete $(d - 1)$-dimensional complex on $n$ vertices. 
	\item While there are fewer than $\binom{n-1}{d}$ faces do:
	\begin{enumerate}
		\item Pick $\sigma$ uniformly at random from among all $d$-dimensional faces in $\Delta_{n-1}$ not yet considered. If $\sigma$ does not complete a cycle in the top homology group with $\Q$-coefficients of the complex so far, add it to the complex. 
		\item Otherwise, do not add $\sigma$ to the complex.
	\end{enumerate}
	\item Return the complex. 
\end{enumerate}

This algorithm is the higher-dimensional analogue of Kruskal's algorithm for finding a minimal-weight spanning tree in an Erd\H{o}s--R\'{e}nyi random graph process with weights indexing the random order in which the edges are added. Even though Kruskal's algorithm classically refers to an algorithm for finding a spanning tree, here it makes sense to consider it as an algorithm for generating a random tree. It is important to note that this algorithm in general will, even in the 1-dimensional case, \emph{not} return a uniform spanning hypertree. 

The complex $Y^3_8$ was found by running 10,000 trials of Kruskal's algorithm. It was one of two examples generated which was collapsible but not anticollapsible.

A later attempt using Kruskal's algorithm with 100,000 runs with $n = 8$ and $d = 2$ also yielded an example of a collapsible but not anticollapsible hypertree which could be used for the base case in place of $Y^2_8$.

The collapsibility of $Y^2_8$ can be proved by hand. The simplicial complex in Figure \ref{fig:3} is a subcomplex of $Y^2_8$ and is clearly collapsible since the edge [3, 6] is free and after we enter the square we can easily collapse away all the triangles. The only thing left to check is that $Y^2_8$ collapse to this complex but this can also easily be done in only six collapsing steps. We leave the details to the reader. 

Alternatively, the reader may use, for example, the Random Discrete Morse algorithm implementation in \polymake \cite{polymake:2000} to verify that $Y_8^2$ and $Y_8^3$ are collapsible, but are dual to non-collapsible complexes, in particular not a single anticollapsing move is possible. 

The duals of these two examples and the triangulation of the dunce hat in Figure \ref{fig:1} gives us a proof of the following:

\begin{proposition}
 There exist simplicial complexes with $8$ vertices in dimension $2,3$ and $4$ that are anticollapsible, and so in particular are contractible, but with no free faces.
\end{proposition}

\subsection{Induction}

We now want to prove the inductive step. That is, given a $d$-dimensional on $n$ vertices simplicial complex $X$, which is anticollapsible and non-collapsible, we want to construct $X'$ which is $(d + 1)$-dimensional on $(n + 1)$ vertices while still being anticollapsible and non-collapsible. To do so we need the following construction.

\begin{definition}\label{def:Ind}
 Let $X$ be a simplicial complex of dimension $d$ on $n$ vertices $(x_1,\ldots ,x_n)$ and let $x := x_i$ be one of them. Given a label $a$ we will denote by $X_{x,a}$ the simplicial complex $X$ where the vertex $x$ is labeled by $a$. We then define:
 \begin{equation}
  X^1_{x}=\{[a]\} * X_{x,b} \cup \{[b]\} * X_{x,a}
 \end{equation}
Where $*$ is the join of two complexes. The faces of the join are the union of a face of the first complex and a face of the second one. 

$X^1_{x}$ is a simplicial complex on $n+1$ vertices $(a,b,x_1, \ldots, x_{i-1}, x_{i+1}, \ldots, x_n)$  of dimension $d+1$.
\end{definition}

\begin{figure}\label{figure:2}
\centering
 \includegraphics[width=12cm]{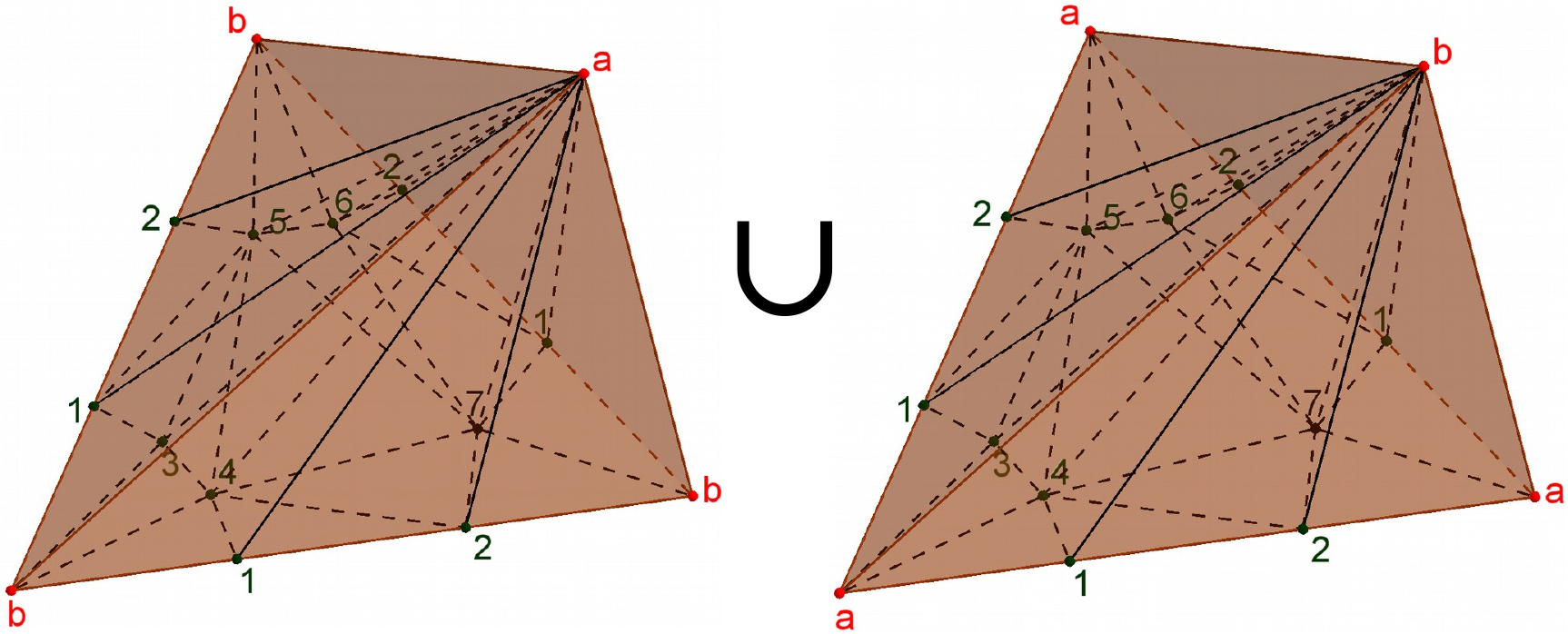}
 \captionsetup{width=.9\linewidth}
 \caption{Construction of Definition \ref{def:Ind} applied to the dunce hat triangulation of Figure~\ref{fig:1} with $0$ as the special vertex.}
\end{figure}

 This construction has also a nice presentation in the Alexander dual, in particular there is a bijection between the facets of $X^*$ and the facets of $X^{1*}_x$, where a cell $\sigma$ is sent to itself if it does not contain $x$, otherwise if it is of the form $\sigma=[x,\sigma']$ it is sent to $[a,b,\sigma']$.

 We are now going to show that many interesting properties are preserved while going from $X$ to $X^1_x$, especially those we are interested in: contractibility, non-collapsibility, and anticollapsibility. 

\begin{lemma}
 Let $X$ be a $d$-dimensional simplicial complex with no free faces, then for any vertex $x$ of $X$, we have that $X^1_x$ has no free faces.
\end{lemma}

\begin{proof}

Let $\tau$ be a $d$-dimensional face of $X^1_x$. There are then three possible cases:
\begin{itemize}
 \item $a \in \tau$, then $\tau=[a,\tau']$ and $\tau'$ is a $(d-1)$-dimensional face of $X_{x,b}$, in particular it is contained in at least two facets $\sigma$ and $\sigma'$, which implies that $\tau$ is contained in $[a,\sigma]$ and $[a,\sigma']$.
 
 \item $b \in \tau$, which is exactly the same as above.
 
 \item $a,b \notin \tau$, but this clearly implies that $\tau$ is contained in $[a,\tau]$ and $[b,\tau]$, so is not a free face.
\end{itemize}

\end{proof}

%
%
%
%

We turn now to discrete Morse theory and, given an acyclic matching on a simplicial complex $X$, we would like to lift it to $X_x^1$. We will do this in two steps. First, recall that, by definition, we have that $\link(a,X^1_x)=X_{x,b}$. Then, since $X_{x,b}$ is combinatorially isomorphic to $X$, we start by lifting the entire matching to the cells that contain $a$; i.e. given a matching pair $(\tau,\sigma)$ in $X_{x,b}$ we add the pair $([a,\tau],[a,\sigma])$ to our newly defined matching in $X^1_x$. We could now be tempted to do the same with respect to $b,$ but it can be easily seen that in this way we will obtain something not well defined. Instead what we do is to look at the restriction of the initial matching to $\del(x,X)$ and lift it to the cells that do not contain $a$. We describe this construction formally below. 

\smallskip
\noindent
\textbf{Construction of a matching on $X^1_x$}.
Given an acyclic matching $M$ on $X$ and a vertex $x$, we will call by $M_b$ the same matching on $X_{x,b}$. We then construct a matching $M_x^1$ on $X_x^1$ in the following way.

Let $(\tau,\sigma) \in M_{b}$ be a matching pair with $\tau \subset \sigma$, then:
\begin{itemize}
 \item $([a,\tau],[a,\sigma]) \in M_x^1$,
 
 \item If $b \notin \tau$  then $(\tau,\sigma) \in M_x^1$,
 
 \item If $b \notin \sigma$ then $([b,\tau],[b,\sigma]) \in M_x^1$.
\end{itemize}

\begin{lemma}
 The matching defined above is acyclic and, if the critical cells of the matching on $X$ forms a subcomplex $Y$ then, the critical cells of the lifted matching on $X_x^1$ are exactly the 
cells of $Y^1_x$.
\end{lemma}

\begin{proof}

First of all, by construction, we immediately obtain that the collection of edges defined above is a matching. 

The fact that it is acyclic follows from the Patchwork Theorem \ref{teo:patchwork} where $Q=\{0,1\}$ and the poset map is the map that sends a cell to $1$ if it contains $a$ and to $0$ otherwise. This is clearly a well-defined poset map and the matching can be restricted to the fibers, so proving that our matching is acyclic is equivalent to proving that the matching restricted to each fiber is acyclic. The matching on the fiber of $1$ is clearly acyclic because it is equivalent to the starting matching on $X$. We need now to check that the matching on the fiber of $0$, i.e. the matching restricted to the cells that do not contain $a$, is acyclic. We are going to prove this by contradiction.

Let $\sigma_0 \searrow \tau_0 
 \nearrow \sigma_1 \searrow \tau_1 \cdots \nearrow \sigma_k=\sigma_0$ be a cycle in the directed Hasse diagram of $X^1_x$, i.e. for each $i$, $(\tau_i,\sigma_i)$ is a pair in the matching while $(\tau_i, \sigma_{i+1})$ is not a pair in the matching, but $\tau_i$ is a face of $\sigma_{i+1}$.
 
  Let $\sigma_i'$ and $\tau_i'$ be the restrictions of these cells 
 to the vertices different from $b$. By construction we obtain that for each $i$ the pair
 $(\tau_i',\sigma_i')$ is a matched pair in $X$ or $\sigma_i'=\tau_i'$ and 
 equivalently $\tau_i'$ is a face of $\sigma_{i+1}'$ and is not paired with it or the two cells are equal.
 
 Then the restriction is still a cycle in $X$. But since the matching on $X$ is 
 acyclic we must have that all the restrictions are equal to $\sigma_0'$ which is impossible.

 Let us now suppose that the critical cells of the matching on $X$ forms a subcomplex $Y$.
 
 Let $\sigma$ be a cell of $X^1_x$, we will show that $\sigma$ is critical if and only if it belongs to $Y_x^1$. To do so we need to analyze various cases separately.
 \begin{itemize}
  \item $a \in \sigma$. Let us then write $\sigma=[a,\sigma']$. The following chain of implications is true:
  
  \centerline{$\sigma$ is critical in $M_x^1$ $\Leftrightarrow$ $\sigma'$ is critical in $M_b$ $\Leftrightarrow$ $\sigma' \in Y_{x,b}$ $\Leftrightarrow$ $\sigma \in Y_x^1$.}
  
  \item $a \notin \sigma$, $b \in \sigma$. As before let us write $\sigma=[b,\sigma']$ and we obtain the exact same chain:
  
  \centerline{$\sigma$ is critical in $M_x^1$ $\Leftrightarrow$ $\sigma'$ is critical in $M_b$ $\Leftrightarrow$ $\sigma' \in Y_{x,b}$ $\Leftrightarrow$ $\sigma \in Y_x^1$.}
  
  \item $a,b \notin \sigma$. This last case again follow from a simple chain of implications:
  
    \centerline{$\sigma$ is critical in $M_x^1$ $\Leftrightarrow$ $\sigma$ is critical in $M_b$ $\Leftrightarrow$ $\sigma \in Y_{x,b}$ $\Leftrightarrow$ $\sigma \in Y_x^1$.}
 \end{itemize}

\end{proof}

We should notice that, while lifting the matching, we do not really need for $x$ to be a vertex of $X$ or $Y$. In the case $x \notin Y$ by $Y^1_x$ we mean, with a slight abuse of notation, the double cone over $Y$ on the new vertices $a$ and $b$, and the same if $x \notin Y$. This can be observed if we recall that $Y^1_x$ is a union of $\{a\} * Y_{x, b}$ and $\{b\} * Y_{x, a}$ where $Y_{x, a}$ is the labeled complex resulting from relabeling vertex $x$ as $a$, and likewise for $Y_{x, b}$.
 The previous lemma is still true in these special cases.

Using this newly constructed matching and simple homotopy theory we are now able to show that our construction preserves contractibility.

\begin{corollary}\label{contractiblelift}
 Given a contractible simplicial complex $X$ and any $x \in X$ we have that $X^1_x$ is contractible.
\end{corollary}

\begin{proof}
 By [Thm. 21 \cite{whitehead1939simplicial}] any contractible simplicial complex can be reduced to a point by a sequence of elementary collapses and anticollapses. Let $X=X_0 \downarrow Y_1 \uparrow X_1 \downarrow \ldots \downarrow Y_k=\{v\}$ be one of such sequences where by $X_{i-1} \downarrow Y_i$ we mean that $X_{i-1}$ collapses to $Y_i$, while by $Y_i \uparrow X_i$ that $Y_i$ anticollapses to $X_i$. Each step of these sequences is in particular an acyclic matching on a $X_i$. We can then use the lifting of the matching defined above and obtain that for any $x \in X$, $X_x^1$ is homotopy equivalent to $(Y_k)_x^1$. If $v=x$ then $(Y_k)_x^1$ is a segment; otherwise, with the same abuse of notation already discussed, $(Y_k)_x^1$ is the union of two segments attached at one vertex. In both cases $(Y_k)_x^1$ is contractible, which means that $X_x^1$ is contractible. 
\end{proof}

\begin{remark}
It can be shown that for any nonempty complex $X$, the complex $X_x^1$ satisfies for all $i \geq 0$, $H_{i + 1}(X_x^1) = H_i(X)$ and $H_0(X_x^1) = 0$. Recall that $X_x^1 = \{[a]\} * X_{x, b} \cup \{[b]\}*X_{x, a}$. Therefore $X_x^1$ is a union of two contractible complexes whose intersection is $\del(x, X) \cup (\{a, b\} * \link(x, X))$. From this a routine Mayer--Vietoris argument may be applied to show the shift in homology from $X$ to $X_x^1$. We omit the proof as it isn't necessary to the discussion of contractible complexes.

 Many other properties of $X$ are also preserved by  $X^1_x$, for example non-evasiveness. We do not prove non-evasiveness is preserved here, but it is easy to check.
\end{remark}

\begin{lemma}\label{Cor:BigDim}
 Let $X$ be a simplicial complex on $n$ vertices of dimension $d$ without free faces. If $X$ 
anticollapses to $\Delta_{n-1}$, then for any vertex $x \in X$, $X^1_x$ is a simplicial 
complex on $n+1$ vertices of dimension $d+1$ without free faces that anticollapses to the 
simplex $\Delta_{n}$.
\end{lemma}

\begin{proof}
 The statement follows immediately from the previous lemmas. Notice that, for any $n \in \N$, and any vertex $x \in \Delta_{n-1}$, $(\Delta_{n-1})^1_x$ is combinatorially isomorphic to $\Delta_{n}$.
\end{proof}

Lemma \ref{Cor:BigDim} is the main inductive tool to prove Theorem \ref{theorem2}. Namely, Lemma \ref{Cor:BigDim} tells use that if Theorem \ref{theorem2} holds for $(n, d)$ then it holds for $(n + 1, d + 1)$.


We almost have the full proof of Theorem \ref{theorem2}. To finish the proof we will show that if the theorem holds for $(n, 2)$ then it holds also for $(n + 1, 2)$. Fortunately this can be easily accomplished by the following proposition.


\begin{proposition}\label{Prop:SameDim}
 If $X$ is a  simplicial complex on $n$ vertices of dimension $d$ without free faces that 
anticollapses to the simplex $\Delta_{n-1}$, then $Y$ obtained from $X$ by deleting a facet and adding the cone over its boundary is anticollapsible to the simplex $\Delta_{n}$ and has no free faces.
\end{proposition}

\begin{proof}
It is obvious that the complex $Y$ still has no free faces.
 
 We now check anticollapsiblity. Let $v$ be the new vertex of $Y$ and $\sigma$ the facet of $X$ that we have deleted. By construction we can perform the elementary anticollapse $([\sigma], [v,\sigma])$ and call $Y'$ the new complex obtained. We now have that $\del(v,Y')=X$, and since $X$ anticollapses to $\Delta_{n-1}$ we can perform the same anticollapsing moves to $Y'$ obtaining a new complex $Y''$. Now $\del(v,Y'')=\Delta_{n-1}$ and $\link(v,Y'')=\sigma$ which are both non-evasive. In particular $Y''$ is non-evasive and therefore anticollapsible.
\end{proof}

From this we immediately obtain Theorem \ref{theorem2} as Proposition \ref{Prop:SameDim} implies that if Theorem \ref{theorem2} holds for $(n, 2)$ then it holds $(n + 1, 2)$, and the dunce hat on 8 vertices gives the base case to this induction, and the other cases have already been proved. 

\begin{remark}\label{rmk:bistellar}
The change to a complex described in the proof of Proposition \ref{Prop:SameDim} is called a bistellar-0 flip or a stacking move. Moves of this type are described in \cite{PachnerDe, PachnerEn}.
\end{remark}

\section{Conclusions}

A motivation to consider the topic discussed here comes from computational experiments of Joswig, Lutz, Lofano, and Tsuruga \cite{joswig2014sphere} in the context of sphere recognition. The authors examine a \emph{random} collapse procedure on simplices of increasing dimension and quite surprisingly it seems, at least empirically, that the probability to get stuck increases exponentially as $n \rightarrow \infty$. For example their experiments showed that in only 12 out of one billion attempts did a random collapsing sequence fail to reach a single vertex on the 8-simplex, but on the 16-simplex 4 trials out of 10,000 failed to reach a single vertex. In their largest example, 46 out of 50 collapsing sequences on the 25-simplex failed. Here we characterize where a collapsing sequence can get stuck, but the expected behavior of a random collapsing procedure remains unclear.


Questions about random collapses of the simplex are also related to the questions about hypertree enumeration. Rather than analyzing random collapses, one could take a uniform distribution over all $d$-complexes that the $n$-simplex can collapse to and ask how many are collapsible. This is a special case of the problem of enumerating different types of hypertrees. Questions of this type appear to be quite difficult. For $d = 1$, there is the well known enumeration commonly called Cayley's formula that show that there are $n^{n - 2}$ labeled trees on $n$ vertices. However, for larger values of $d$ it is not even known how many $d$-dimensional hypertrees on $n$ vertices there are. The closest we have to an enumeration formula is the following classic result of Kalai giving a weighted enumeration formula.

\begin{theorem}[Kalai \cite{Kalai1983}]
For $n, d \geq 1$ let $\mathcal{T}_{n, d}$ denote the collection of $d$-dimensional $\Q$-acyclic complexes on $n$ vertices with complete $(d - 1)$-skeleton, then 
$$\sum_{X \in \mathcal{T}_{n, d}} |H_{d - 1}(X)|^2 = n^{\binom{n - 2}{d}}.$$
\end{theorem}

This weighted enumeration formula was later extended by Duval, Klivans, and Martin \cite{DuvalKlivansMartin2009} to the case where the $(d - 1)$-skeleton need not be complete in analogy to how Kirchhoff's Matrix Tree theorem generalizes Cayley's formula. Unweighted enumeration remains an open problem, with the best known bounds given by Linial and Peled \cite{LinialPeled2019}. In \cite{LinialPeled2019} the authors conjecture that almost all $d$-dimensional hypertrees, with complete $(d - 1)$-skeleton, are not $d$-collapsible for $d \geq 2$. Similar conjectures could be made for other properties of hypertrees. For instance is it true that almost all collapsible $d$-dimensional hypertrees fail to be anticollapsible for $d \geq 2$?


Our constructions give examples of complexes that are anticollapsible but not collapsible. Naturally, one could ask for contractible complexes that are \emph{neither} anticollapsible nor collapsible. The example $C_3^8$ below is such a complex. This example was found by using Kruskal's algorithm to generate 500,000 examples of 3-dimensional hypertrees on 8 vertices. We point out that for $d \geq 3$ any $d$-dimensional $\Z$-acyclic hypertree with complete 2-skeleton is necessarily contractible by the Whitehead theorem. In our 500,000 runs $C_3^8$ was one of three examples that was neither 3-collapsible nor 3-anticollapsible.

\smallskip
$C_3^8$ = \{ [ 1, 5, 7, 8 ], [ 3, 4, 5, 8 ], [ 1, 2, 6, 7 ], [ 1, 2, 3, 5 ], [ 1, 3, 4, 6 ], [ 2, 4, 7, 8 ], [ 4, 5, 6, 7 ],  [ 2, 3, 7, 8 ], [ 1, 3, 5, 6 ], [ 2, 4, 5, 8 ], [ 1, 3, 4, 8 ], [ 2, 3, 4, 5 ], [ 1, 2, 4, 6 ], [ 2, 4, 6, 7 ],  [ 2, 4, 5, 7 ], [ 1, 3, 5, 7 ], [ 1, 3, 4, 5 ], [ 2, 3, 6, 7 ], [ 3, 5, 7, 8 ], [ 3, 4, 5, 7 ], [ 1, 3, 4, 7 ],  [ 2, 3, 6, 8 ], [ 2, 3, 4, 6 ], [ 1, 3, 7, 8 ], [ 1, 5, 6, 7 ], [ 2, 5, 6, 8 ], [ 4, 6, 7, 8 ], [ 1, 5, 6, 8 ],  [ 2, 3, 5, 6 ], [ 1, 2, 3, 8 ], [ 3, 4, 6, 8 ], [ 1, 2, 5, 7 ], [ 1, 2, 4, 8 ], [ 5, 6, 7, 8 ], [ 3, 4, 6, 7 ] \}

\smallskip

Non-collapsibility and non-anticollapsiblity for this complex may be verified using the implementation of the Random Discrete Morse algorithm in \polymake. For both $C_3^8$ and $(C_3^8)^*$ the Random Discrete Morse algorithm returns a Morse vector with a critical cell in the top dimension.


Even without this example, one can establish that contractible complexes which are neither collapsible nor anticollapsible must exist. Indeed all possible collapsible sequences from a complex on $n$ vertices are finite, but deciding if a complex is contractible is undecidable as it requires deciding if the complex has trivial fundamental group \cite{tancer2016recognition}.

While $C^8_3$ was one of only three examples out of 500,000 randomly generated 3-dimensional trees to be contractible but neither anticollapsible nor collapsible, we suspect that for large $n$ almost all contractible complexes are neither collapsible nor anticollapsible.

\section*{Acknowledgment}
Both authors thank Frank H. Lutz for suggesting this problem and for several helpful comments on an earlier draft of the article. 

\bibliographystyle{amsalpha}
\bibliography{bibliography}

\end{document}